\documentclass[reqno]{amsart}

\usepackage{amssymb}
\usepackage{graphicx}
\usepackage{amscd}
\usepackage[hidelinks]{hyperref}
\usepackage{color}
\usepackage{float}
\usepackage{graphics,amsmath,amssymb}
\usepackage{amsthm}
\usepackage{amsfonts}
\usepackage{latexsym}
\usepackage{epsf}
\usepackage{xifthen}
\usepackage{mathrsfs}
\usepackage{dsfont}
\usepackage{makecell}
\usepackage[FIGTOPCAP]{subfigure}
\usepackage{amsmath}
\allowdisplaybreaks[4]
\usepackage{listings}
\usepackage{etoolbox}
\usepackage{fancyhdr}
\usepackage{pdflscape}
\usepackage[title,toc,titletoc]{appendix}
\usepackage{enumitem}
\usepackage[noadjust]{cite}

 \newtheoremstyle{mytheorem}% name % cf. thmtest.tex of AMSLaTeX
 {3pt}%      Space above
 {3pt}%      Space below
 {\slshape}% Body font
 {}%         Indent amount (empty = no indent,
 {\bfseries}% Thm head font
 {.}%        Punctuation after thm head
 { }%        Space after thm head: " " = normal interword space;
 {}%         Thm head spec (can be left empty, meaning `normal')

\numberwithin{equation}{section}

\theoremstyle{theorem}
\newtheorem{theorem}{Theorem}[section]
\newtheorem*{theorem*}{Theorem}
\newtheorem{corollary}[theorem]{Corollary}

\newtheorem{proposition}[theorem]{Proposition}

\theoremstyle{definition}

\newtheorem{remark}{Remark}[section]
\newtheorem*{remark*}{Remark}
\newtheorem*{remarks*}{Remarks}

\newtheoremstyle{named}{}{}{\itshape}{}{\bfseries}{.}{.5em}{#1\thmnote{ #3}}
\theoremstyle{named}

\newcommand{\Keywords}[1]{\ifthenelse{\isempty{#1}}{}{\smallskip \smallskip \noindent \textbf{Keywords}. #1}}
\newcommand{\MSC}[2][2010]{\ifthenelse{\isempty{#2}}{}{\smallskip \smallskip \noindent \textbf{#1MSC}. #2}}
\newcommand{\abstractnote}[1]{\ifthenelse{\isempty{#1}}{}{\smallskip \smallskip \noindent \textsuperscript{\dag}#1}}

\makeatletter
\def\specialsection{\@startsection{section}{1}%
  \z@{\linespacing\@plus\linespacing}{.5\linespacing}%
  {\normalfont}}% NEW
\def\section{\@startsection{section}{1}%
  \z@{.7\linespacing\@plus\linespacing}{.5\linespacing}%
  {\normalfont\scshape}}% NEW
\patchcmd{\@settitle}{\uppercasenonmath\@title}{\Large\boldmath}{}{}
\patchcmd{\@settitle}{\begin{center}}{\begin{flushleft}}{}{}
\patchcmd{\@settitle}{\end{center}}{\end{flushleft}}{}{}
\patchcmd{\@setauthors}{\MakeUppercase}{\normalsize}{}{}
\patchcmd{\@setauthors}{\centering}{\raggedright}{}{}
\patchcmd{\section}{\scshape}{\large\bfseries\boldmath}{}{}
\patchcmd{\subsection}{\bfseries}{\bfseries\boldmath}{}{}
\renewcommand{\@secnumfont}{\bfseries}
\patchcmd{\@startsection}{\@afterindenttrue}{\@afterindentfalse}{}{}
\patchcmd{\abstract}{\leftmargin3pc}{\leftmargin1pc}{}{}

\def\maketitle{\par
  \@topnum\z@ % this prevents figures from falling at the top of page 1
  \@setcopyright
  \thispagestyle{empty}% this sets first page specifications
  \ifx\@empty\shortauthors \let\shortauthors\shorttitle
  \else \andify\shortauthors
  \fi
  \@maketitle@hook
  \begingroup
  \@maketitle
  \toks@\@xp{\shortauthors}\@temptokena\@xp{\shorttitle}%
  \toks4{\def\\{ \ignorespaces}}% defend against questionable usage
  \edef\@tempa{%
    \@nx\markboth{\the\toks4
      \@nx\MakeUppercase{\the\toks@}}{\the\@temptokena}}%
  \@tempa
  \endgroup
  \c@footnote\z@
  \@cleartopmattertags
}
\makeatother

\newcommand{\mex}{\operatorname{mex}}
\newcommand{\smex}{\sigma\!\operatorname{mex}}
\newcommand{\maex}{\operatorname{maex}}
\newcommand{\smaex}{\sigma\!\operatorname{maex}}
\newcommand{\sL}{\sigma\!L}

\title{Partitions and the maximal excludant}

\author[S. Chern]{Shane Chern}
\address{Department of Mathematics, Penn State University, University Park, PA 16802, USA}
\email{shanechern@psu.edu}

\date{}

\begin{document}

%{\footnotesize\noindent To appear in \textit{Bull. Aust. Math. Soc.}\\
%Available at \arxiv{1812.04657}.}
%
%\bigskip \bigskip

\maketitle

\begin{abstract}

For each nonempty integer partition $\pi$, we define the maximal excludant of $\pi$ to be the largest nonnegative integer smaller than the largest part of $\pi$ that is not a part of $\pi$. Let $\sigma\!\operatorname{maex}(n)$ be the sum of maximal excludants over all partitions of $n$. We show that the generating function of $\sigma\!\operatorname{maex}(n)$ is closely related to a mock theta function studied by Andrews \textit{et al.} and Cohen. Further, we show that, as $n\to \infty$, $\sigma\!\operatorname{maex}(n)$ is asymptotic to the sum of largest parts of all partitions of $n$. Finally, the expectation of the difference of the largest part and the maximal excludant over all partitions of $n$ is shown to converge to $1$ as $n\to \infty$.


\MSC{05A17, 05A19, 11P84.}
\end{abstract}

\section{Introduction}

In a recent paper \cite{AN2019}, Andrews and Newman studied the \textit{minimal excludant} of an integer partition $\pi$, which is the smallest positive integer that is not a part of $\pi$. Since a nonempty partition $\pi$ is a finite sequence of positive integers, we may also study the \textit{maximal excludant} of $\pi$, by which we mean the largest nonnegative integer smaller than the largest part of $\pi$ that is not a part of $\pi$. For example, $5$ has seven partitions: $5$, $4+1$, $3+2$, $3+1+1$, $2+2+1$, $2+1+1+1$ and $1+1+1+1+1$, the maximal excludants of which are respectively $4$, $3$, $1$, $2$, $0$, $0$ and $0$.

Let $\mex(\pi)$ and $\maex(\pi)$ denote respectively the minimal and maximal excludant of $\pi$. Andrews and Newman further investigated the function
$$\smex(n):=\sum_{\pi\vdash n}\mex(\pi)$$
in which the summation is over all partitions of $n$. They proved that the generating function of $\smex(n)$ satisfies
\begin{equation}
\sum_{n\ge 0}\smex(n)q^n=(-q;q)_\infty^2,
\end{equation}
where we adopt the conventional $q$-Pochhammer symbols:
\begin{gather*}
(A;q)_n:=\prod_{k=0}^{n-1}(1-Aq^k)\\
\intertext{and}
(A;q)_\infty:=\lim_{n\to \infty}(A;q)_n.
\end{gather*}
Likewise, we may define another function
$$\smaex(n):=\sum_{\pi\vdash n}\maex(\pi)$$
where, again, the summation runs over all partitions of $n$. In this paper, we are to study the generating function of $\smaex(n)$. As we shall see in Theorem \ref{th:main}, unlike the generating function of $\smex(n)$, which is a weight $0$ modular form, the generating function of $\smaex(n)$ is closely related to a mock theta function studied in two side-by-side papers of Andrews \textit{et al.}~\cite{ADH1988} and Cohen \cite{Coh1988}.

\begin{theorem}\label{th:main}
	We have
	\begin{align}
		\sum_{n\ge 1}\smaex(n)q^n&=\sum_{k\ge 1}\frac{k}{(q;q)_{k-1}}\sum_{m\ge 1}q^{m(k+1)}(-q;q)_{m-1}\label{eq:main-1}\\
		&=\frac{1}{(q;q)_\infty}\Bigg(\sum_{n\ge 1}\frac{q^n}{1-q^n}-\sum_{n\ge 1}q^n (q^2;q^2)_{n-1}\Bigg)\label{eq:main-2}\\
		&=\frac{1}{(q;q)_\infty}\Bigg(\sum_{n\ge 1}\frac{q^n}{1-q^n}+\sum_{n\ge 1}\frac{(-1)^n q^{n^2}}{(q;q^2)_n}\Bigg).\label{eq:main-3}
	\end{align}
\end{theorem}

\begin{remark}
	Using a formula due to Andrews \textit{et al.}~\cite{ADH1988}, we are able to give a formula of $\smaex(n)$. This will be discussed in \S{}\ref*{sec:formula}.
\end{remark}

Now recall that if $L(\pi)$ denotes the largest part of a partition $\pi$ and $$\sL(n)=\sum_{\pi\vdash n}L(\pi)$$
denotes the sum of largest parts of all partitions of $n$, a standard result tells us that
\begin{equation}\label{eq:largest-pt}
\sum_{n\ge 1}\sL(n)q^n=\frac{1}{(q;q)_\infty}\sum_{n\ge 1}\frac{q^n}{1-q^n}.
\end{equation}
In light of \eqref{eq:main-2} and \eqref{eq:largest-pt}, we have the following corollary.

\begin{corollary}
	We have
	\begin{equation}\label{eq:dif}
	\sum_{n\ge 1}\big(\sL(n)-\smaex(n)\big)q^n=\frac{1}{(q;q)_\infty}\sum_{n\ge 1}q^n (q^2;q^2)_{n-1}.
	\end{equation}
\end{corollary}

It was shown by Kessler and Livingston \cite{KL1976} that $\sL(n)$ satisfies the asymptotic formula
\begin{equation}\label{eq:as-0}
\sL(n)\sim \frac{\log \frac{6n}{\pi^2}+2\gamma}{4\pi \sqrt{2n}}e^{\pi \sqrt{\frac{2n}{3}}}
\end{equation}
where $\gamma$ is the Euler--Mascheroni constant.

Now we shall show the asymptotic relations as follows.

\begin{theorem}\label{th:asym}
	We have, as $n\to\infty$,
	\begin{equation}\label{eq:as-1}
	\sL(n)-\smaex(n)\sim \frac{1}{4\sqrt{3}n}e^{\pi \sqrt{\frac{2n}{3}}},
	\end{equation}
	and, a fortiori,
	\begin{equation}\label{eq:as-2}
	\smaex(n)\sim \sL(n).
	\end{equation}
	Further, if $\mathbb{E}_n$ denotes the expectation of the difference of the largest part and the maximal excludant over all partitions of $n$, then
	\begin{equation}\label{eq:as-3}
	\lim_{n\to\infty}\mathbb{E}_n = 1.
	\end{equation}
\end{theorem}

\begin{remark}
	Notice that for any nonempty partition $\pi$, we always have $L(\pi)-\maex(\pi)\ge 1$. Hence, for all $n\ge 1$, we have $\mathbb{E}_n\ge 1$. Further, for all $n\ge 3$, it is always able to find a partition $\pi$ of $n$ with $L(\pi)-\maex(\pi)> 1$ (if $n$ is odd, then such a partition could be $((n+1)/2,(n-1)/2)$; if $n$ is even, then such a partition could be $(n/2,(n-2)/2,1)$). This implies that $\mathbb{E}_n> 1$ for $n\ge 3$.
\end{remark}

\section{A formula of $\smaex(n)$}\label{sec:formula}

\subsection{A mock theta function}

In his paper \cite{Coh1988}, Cohen observed the following identity
\begin{equation}\label{eq:Coh}
\sum_{n\ge 1}\frac{(-1)^n q^{n^2}}{(q;q^2)_n}=-\sum_{n\ge 1}q^n (q^2;q^2)_{n-1}.
\end{equation}
Hence, \eqref{eq:main-2} and \eqref{eq:main-3} are equivalent.

It is worth mentioning that Cohen's identity \eqref{eq:Coh} can be generalized to a three-variable identity as follows.

\begin{proposition}\label{prop:3-var}
	We have
	\begin{equation}\label{eq:3-var}
	\sum_{n\ge 1}\frac{x^n q^{n^2}}{(yq;q^2)_n}=\sum_{n\ge 1}x y^{n-1}q^n(-xq^2/y;q^2)_{n-1}.
	\end{equation}
\end{proposition}
Taking $x=-1$ and $y=1$ in \eqref{eq:3-var} recovers \eqref{eq:Coh}. Further, this identity can be treated as a companion to \cite[p.~29, Example 6]{And1976}:
$$\sum_{n\ge 0}\frac{x^n q^{n^2}}{(y;q^2)_{n+1}}= \sum_{n\ge 0} y^n (-xq/y;q^2)_n.$$

\begin{proof}[Proof of Proposition \ref{prop:3-var}]
	Both sides of \eqref{eq:3-var} can be treated as the generating function of partitions in which the largest part appears only once and all the remaining distinct parts appear exactly twice. Here, the exponent of $x$ represents the number of distinct parts in this partition and the exponent of $y$ represents the largest part minus the number of distinct parts.
\end{proof}

\begin{remark}
	Let us denote
	\begin{align}\label{eq:sigma*}
	\sigma^*(q):=2\sum_{n\ge 1}\frac{(-1)^n q^{n^2}}{(q;q^2)_n}=-2\sum_{n\ge 1}q^n (q^2;q^2)_{n-1}.
	\end{align}
	It is also necessary to introduce its companion
	\begin{align}\label{eq:sigma}
	\sigma(q):=\sum_{n\ge 0}\frac{q^{n(n+1)/2}}{(-q;q)_n}=1-\sum_{n\ge 1}(-1)^n q^n (q;q)_{n-1}.
	\end{align}
	The two $q$-hypergeometric functions are of substantial research interest along the following lines. First, Andrews \textit{et al.}~\cite{ADH1988} showed that the coefficients in the expansions of $\sigma(q)$ and $\sigma^*(q)$ are very small. In fact, these coefficients are related with the arithmetic of the field $\mathbb{Q}(\sqrt{6})$. Second, let us define a sequence $\{T(n)\}_{n\in 24\mathbb{Z}+1}$ by
	\begin{equation}\label{eq:T}
	q\sigma(q^{24})=\sum_{n\ge 0}T(n)q^{n} \quad\text{and}\quad q^{-1}\sigma^{*}(q^{24})=\sum_{n<0}T(n)q^{-n}.
	\end{equation}
	Cohen \cite{Coh1988} proved that the function (in which $K_0(x)$ is the Bessel function)
	$$\phi_0(\tau):=y^{1/2}\sum_{n\in 24\mathbb{Z}+1} T(n)K_0(2\pi |n|y/24)e^{2\pi i n x/24}\qquad(\tau=x+iy\in\mathbb{H})$$
	is a Maass wave form on the congruence group $\Gamma_0(2)$. This, in turn, explains the modularity nature of the identity
	$$q\sigma(q^{24})=\sum_{\substack{a,b\in\mathbb{Z}\\a>6|b|}}\left(\frac{12}{a}\right)(-1)^b q^{a^2-24b^2}.$$
	Third, by noticing the following relation due to Cohen \cite{Coh1988}:
	\begin{equation}\label{eq:s-s*}
	\sigma(q)=-\sigma^{*}(q^{-1})
	\end{equation}
	whenever $q$ is a root of unity (here the definitions of $\sigma(q)$ and $\sigma^*(q)$ at roots of unity are valid since the second summations in both \eqref{eq:sigma*} and \eqref{eq:sigma} are finite), Zagier \cite{Zag2010} is able to construct a quantum modular form $f:\mathbb{Q}\to\mathbb{C}$ by
	$$f(x):=q^{1/24}\sigma(q)=-q^{1/24}\sigma^{*}(q^{-1})$$
	where $q=e^{2\pi i x}$.
\end{remark}

\subsection{A formula of Andrews \textit{et al.}}

Let us define $T(n)$ ($n\in 6\mathbb{Z}+1$) by the excess of the number of inequivalent solutions of the Pell's equation
$$u^2-6v^2=n$$
with $u+3v\equiv \pm 1 \pmod{12}$ over the number of them with $u+3v\equiv \pm 5 \pmod{12}$. By investigating the arithmetic in $\mathbb{Q}(\sqrt{6})$, Andrews \textit{et al.}~\cite{ADH1988} showed that if $n$ has the prime factorization
$$n=p_1^{e_1}p_2^{e_2}\cdots p_r^{e_r}$$
where each $p_i\equiv 1 \pmod{6}$ or $p_i$ is the negative of a prime $\equiv 5 \pmod{6}$, then
$$T(n)=T(p_1^{e_1})T(p_2^{e_2})\cdots T(p_r^{e_r})$$
where
$$T(p^e)=\begin{cases}
0 & \text{if $p\not\equiv 1 \pmod{24}$ and $e$ is odd},\\
1 & \text{if $p\equiv 13,19 \pmod{24}$ and $e$ is even},\\
(-1)^{e/2} & \text{if $p\equiv 7 \pmod{24}$ and $e$ is even},\\
e+1 & \text{if $p\equiv 1 \pmod{24}$ and $T(p)=2$},\\
(-1)^e(e+1) & \text{if $p\equiv 1 \pmod{24}$ and $T(p)=-2$}.
\end{cases}$$

Andrews \textit{et al.}~further proved that if we restrict $T(n)$ to $n\in 24\mathbb{Z}+1$, then they coincide with the coefficients defined in \eqref{eq:T}. Hence, if we write
$$\sigma^{*}(q)=2\sum_{n\ge 1}\frac{(-1)^n q^{n^2}}{(q;q^2)_n}=2\sum_{n\ge 1}S^*(n)q^n,$$
then
$$S^*(n)=\frac{1}{2}T(1-24n).$$

Further, we have
$$\sum_{n\ge 1}d(n)q^n=\sum_{n\ge 1}\frac{q^n}{1-q^n}$$
where $d(n)=\sum_{d\mid n}1$ enumerates the number of positive divisors of $n$.

Consequently, \eqref{eq:main-3} gives us the following formula of $\smaex(n)$.

\begin{theorem}
	For $n\ge 1$, we have
	\begin{equation}
	\smaex(n)=\sum_{k=1}^n p(n-k)\bigg(d(k)+\frac{1}{2}T(1-24k)\bigg)
	\end{equation}
	where $p(n)$ denotes the number of partitions of $n$.
\end{theorem}

\section{Proof of Theorem \ref{th:main}}

The equivalence of \eqref{eq:main-2} and \eqref{eq:main-3} has already been shown in \S{}\ref{sec:formula}. It suffices to prove \eqref{eq:main-1} and \eqref{eq:main-2}.

Given a partition with maximal excludant $k$, it can be split into two components: the first component is a partition with parts not exceeding $k-1$ and the second component is a gap-free partition (i.e.~a partition in which the difference between each consecutive parts is at most $1$) with smallest part $k+1$. Further, by considering the conjugate, there is a bijection between gap-free partitions with smallest part $k+1$ and partitions in which the largest part repeats $k+1$ times and all remaining parts are distinct. Hence, if $g(k,n)$ counts the number of partitions of $n$ with maximal excludant $k$, we have the generating function identity
\begin{align}
G(z,q):&=\sum_{n\ge 1}\sum_{k\ge 1}g(k,n)z^k q^n\notag\\
&=\sum_{k\ge 1}\frac{z^k}{(q;q)_{k-1}}\sum_{m\ge 1}q^{m(k+1)}(-q;q)_{m-1}.\label{eq:G}
\end{align}

Now applying the operator $[\partial/\partial z]_{z=1}$ directly to $G(z,q)$ implies \eqref{eq:main-1}. Next, we prove \eqref{eq:main-2}. Recall that Euler's first sum \cite[Eq.~(2.2.5)]{And1976} tells us that
$$\sum_{k\ge 0}\frac{z^k}{(q;q)_k}=\frac{1}{(z;q)_\infty}.$$
In light of \eqref{eq:G}, we have
\begin{align*}
G(z,q)&=\sum_{m\ge 1}z q^{2m}(-q;q)_{m-1}\sum_{k\ge 0}\frac{(zq^m)^k}{(q;q)_k}\\
&=\sum_{m\ge 1}\frac{zq^{2m}(q^2;q^2)_{m-1}}{(q;q)_{m-1} (zq^m;q)_\infty}.
\end{align*}

Notice that
\begin{align*}
\Bigg[\frac{\partial}{\partial z}\frac{z}{(zq^m;q)_\infty}\Bigg]_{z=1}&=\Bigg[\frac{z}{(zq^m;q)_\infty}\frac{\partial}{\partial z}\log\frac{z}{(zq^m;q)_\infty}\Bigg]_{z=1}\\
&=\frac{1}{(q^m;q)_\infty}\Bigg[\frac{\partial}{\partial z}\bigg(\log z-\sum_{n\ge m}\log(1-zq^n)\bigg)\Bigg]_{z=1}\\
&=\frac{1}{(q^m;q)_\infty}\Bigg(1+\sum_{n\ge m}\frac{q^n}{1-q^n}\Bigg).
\end{align*}
Hence,
\begin{align*}
\sum_{n\ge 1}\smaex(n)q^n&=\Bigg[\frac{\partial}{\partial z}G(z,q)\Bigg]_{z=1}\\
&=\sum_{m\ge 1}\frac{q^{2m}(q^2;q^2)_{m-1}}{(q;q)_{m-1}}\Bigg[\frac{\partial}{\partial z}\frac{z}{(zq^m;q)_\infty}\Bigg]_{z=1}\\
&=\frac{1}{(q;q)_\infty}\sum_{m\ge 1} q^{2m}(q^2;q^2)_{m-1}\Bigg(1+\sum_{n\ge m}\frac{q^n}{1-q^n}\Bigg).
\end{align*}

An easy combinatorial argument implies that, for all $n\ge 1$, we have
$$\sum_{m=1}^n q^m (q;q)_{m-1}=1-(q;q)_n.$$
It follows that
\begin{align*}
\sum_{n\ge 1}\smaex(n)q^n&=\frac{1}{(q;q)_\infty}\Bigg(1-(q^2;q^2)_\infty+\sum_{n\ge 1}\frac{q^n}{1-q^n}\big(1-(q^2;q^2)_n\big)\Bigg)\\
&=\frac{1}{(q;q)_\infty}\Bigg(1-(q^2;q^2)_\infty+\sum_{n\ge 1}\frac{q^n}{1-q^n}\\
&\quad\quad\quad\quad\quad\quad-\sum_{n\ge 1}q^n(1+q^n)(q^2;q^2)_{n-1}\Bigg)\\
&=\frac{1}{(q;q)_\infty}\Bigg(1-(q^2;q^2)_\infty+\sum_{n\ge 1}\frac{q^n}{1-q^n}\\
&\quad\quad\quad\quad\quad\quad-\big(1-(q^2;q^2)_\infty\big)-\sum_{n\ge 1}q^{n}(q^2;q^2)_{n-1}\Bigg)\\
&=\frac{1}{(q;q)_\infty}\Bigg(\sum_{n\ge 1}\frac{q^n}{1-q^n}-\sum_{n\ge 1}q^n (q^2;q^2)_{n-1}\Bigg).
\end{align*}
This completes the proof of \eqref{eq:main-2}.

\section{Proof of Theorem \ref{th:asym}}

We first show that the nonnegative sequence $\{\sL(n)-\smaex(n)\}_{n\ge 1}$ is weakly increasing. To see this, we construct an injective map $\phi_n:\mathcal{P}_n \to \mathcal{P}_{n+1}$ (where $\mathcal{P}_n$ denotes the set of partitions of $n$) for each $n\ge 1$ by
$$\pi=(\pi_1,\pi_2,\ldots,\pi_\ell)\mapsto (\pi_1,\pi_2,\ldots,\pi_\ell,1).$$
Then $L(\phi_n(\pi))=L(\pi)$ and
$$\maex(\phi_n(\pi))=\begin{cases}
\maex(\pi) & \text{if $\maex(\pi)\ne 1$},
\\0 & \text{if $\maex(\pi)= 1$}.
\end{cases}$$
It follows that
\begin{align*}
\sL(n+1)-\smaex(n+1)&=\sum_{\lambda\vdash n+1}\big(L(\lambda)-\maex(\lambda)\big)\\
&\ge \sum_{\pi\vdash n} \big(L(\phi_n(\pi))-\maex(\phi_n(\pi))\big)\\
&\ge \sum_{\pi\vdash n} \big(L(\pi)-\maex(\pi)\big)\\
&=\sL(n)-\smaex(n).
\end{align*}

It turns out that we may apply Ingham's Tauberian theorem to obtain the asymptotic behavior of $\sL(n)-\smaex(n)$.

\begin{theorem}[Ingham \cite{Ing1941}]\label{th:ing}
	Let $f(q)=\sum_{n\ge0}a(n) q^n$ be a power series with weakly increasing nonnegative coefficients and radius of convergence equal to $1$. If there are constants $A>0$ and $\lambda,\alpha\in\mathbb{R}$ such that
	$$f\left(e^{-t}\right) \sim\lambda t^\alpha e^{\frac{A}{t}}$$
	as $t\to 0^+$, then
	$$a(n)\sim \frac{\lambda}{2\sqrt{\pi}}\frac{A^{\frac{\alpha}{2}+\frac14}}{n^{\frac{\alpha}{2}+\frac34}} e^{2\sqrt{An}}$$
	as $n\to \infty$.
\end{theorem}

Recall from \eqref{eq:dif} and \eqref{eq:sigma*} that
\begin{align*}
\sum_{n\ge 1}\big(\sL(n)-\smaex(n)\big)q^n&=\frac{1}{(q;q)_\infty}\sum_{n\ge 1}q^n (q^2;q^2)_{n-1}\\
&=-\frac{\sigma^*(q)}{2(q;q)_\infty}.
\end{align*}
Now the modular inversion formula for Dedekind's eta-function (p.~121, Proposition 14 of \cite{Kob1984}) implies that, as $t\to 0^+$,
\begin{equation}
\frac{1}{(e^{-t};e^{-t})_\infty}\sim \sqrt{\frac{t}{2\pi}}e^{\frac{\pi^2}{6t}}.
\end{equation}
On the other hand, Zagier \cite{Zag2010} showed that if we take $q=\xi e^{-t}$ with $\xi$ a root of unity, then the identity \eqref{eq:s-s*} remains true as an identity in $\mathbb{Q}[\xi][[t]]$. Taking $\xi=1$, Zagier further obtained the expansion
\begin{equation}
-\sigma^{*}(e^{-t})=2+2t+5t^2+\frac{55}{3}t^3+\frac{1073}{12}t^4+\frac{32671}{60}t^5+\frac{286333}{72}t^6+\cdots
\end{equation}
as $t\to 0$. Hence, as $t\to 0^+$,
$$\Bigg[\sum_{n\ge 1}\big(\sL(n)-\smaex(n)\big)q^n\Bigg]_{q=e^{-t}}\sim \sqrt{\frac{t}{2\pi}}e^{\frac{\pi^2}{6t}}.$$

Finally, \eqref{eq:as-1} follows from Ingham's Tauberian theorem. Further, \eqref{eq:as-2} can be deduced by comparing \eqref{eq:as-1} with \eqref{eq:as-0}. Also, we know that the number of partitions of $n$ satisfies
$$p(n)\sim \frac{1}{4\sqrt{3}n}e^{\pi \sqrt{\frac{2n}{3}}}$$
as $n\to \infty$. Hence,
$$\lim_{n\to\infty}\mathbb{E}_n=\lim_{n\to\infty}\frac{\sL(n)-\smaex(n)}{p(n)}=1.$$

\bibliographystyle{amsplain}

\end{document}